\documentclass[a4paper, 11pt]{amsart}   
\usepackage{mathptmx, amssymb,amscd,latexsym, eulervm}
\usepackage{changepage} 
\usepackage{amsmath}
\usepackage{amsthm}
\usepackage{mathdots}
\usepackage[pagebackref,colorlinks=true,linkcolor=blue,urlcolor=blue]{hyperref}
\usepackage{color}
\usepackage{comment}
\usepackage[onehalfspacing]{setspace}
\usepackage{tabularx}
\usepackage{amsfonts}
\usepackage{paralist}
\usepackage{aliascnt}
\usepackage[initials, lite]{amsrefs}
\usepackage{amscd}
\usepackage{blkarray}
\usepackage{mathbbol}
\usepackage{setspace}
\usepackage[inner=2.4cm,outer=2.4cm, bottom=3.2cm]{geometry}
\usepackage{tikz, tikz-cd}
\usepackage{calligra,mathrsfs}

\usepackage{tikz}
\usetikzlibrary{matrix}
\usetikzlibrary{arrows,calc}
\allowdisplaybreaks

\BibSpec{collection.article}{%
	+{}  {\PrintAuthors}                {author}
	+{,} { \textit}                     {title}
	+{.} { }                            {part}
	+{:} { \textit}                     {subtitle}
	+{,} { \PrintContributions}         {contribution}
	+{,} { \PrintConference}            {conference}
	+{}  {\PrintBook}                   {book}
	+{,} { }                            {booktitle}
	+{,} { }                            {series}
	+{, vol.} { }                            {volume}
	+{,} { }                            {publisher}
	+{,} { \PrintDateB}                 {date}
	+{,} { pp.~}                        {pages}
	+{,} { }                            {status}
	+{,} { \PrintDOI}                   {doi}
	+{,} { available at \eprint}        {eprint}
	+{}  { \parenthesize}               {language}
	+{}  { \PrintTranslation}           {translation}
	+{;} { \PrintReprint}               {reprint}
	+{.} { }                            {note}
	+{.} {}                             {transition}
	+{}  {\SentenceSpace \PrintReviews} {review}
}
\AtBeginDocument{%
	\def\MR#1{}
}

\makeatletter
\@namedef{subjclassname@2020}{%
	\textup{2020} Mathematics Subject Classification}
\makeatother

\newcommand{\bfa}{\mathbf{a}}
\newcommand{\bfb}{\mathbf{b}}

\newcommand{\bfe}{\mathbf{e}}
\newcommand{\bfx}{\mathbf{x}}

\newcommand{\eb}{\mathbf{e}}

\newcommand{\xb}{\mathbf{x}}

\newcommand{\tb}{\mathbf{t}}
\newcommand{\strint}{\mathrm{int}}

\newcommand{\into}{\mathrm{int}}
\newcommand{\verto}{\mathrm{vert}}

\newcommand{\set}[1]{\left\{ #1 \right\}}

\newcommand{\setcond}[2]{\set{#1 \ \colon \ #2}}
\newcommand{\facets}{\mathcal{F}}

\def\opn#1#2{\def#1{\operatorname{#2}}} 
\opn\Cl{Cl} \opn\deg{deg} \opn\Stab{Stab} \opn\aff{aff} \opn\div{div}
\opn\cone{cone} \opn\End{End} \opn\mod{mod}  \opn\pdim{pdim} \opn\diag{diag} \opn\vert{vert} \opn\m{m}
\opn\Cone{Cone} \opn\Pyr{Pyr} \opn\max{max} \opn\min{min} \opn\int{int} \opn\rev{rev} \opn\ker{ker} \opn\lat{lat} \opn\pull{pull}
\opn\cok{cok} \opn\ant{ant}
\opn\inte{int}

\newcommand{\RR}{\mathbb{R}}
\newcommand{\kk}{\mathbb{k}}
\newcommand{\KK}{\mathbb{K}}

\newcommand{\NN}{\normalfont\mathbb{N}}
\newcommand{\ZZ}{\normalfont\mathbb{Z}}

\newcommand{\MM}{{\normalfont\mathfrak{m}}}

\newcommand{\mm}{{\normalfont\mathfrak{m}}}
\newcommand{\fkn}{{\normalfont\mathfrak{n}}}

\newcommand{\QQ}{\mathbb{Q}}
\newcommand{\pp}{{\normalfont\mathfrak{p}}}

\newcommand{\depth}{\normalfont\text{depth}}

\newcommand{\conv}{\normalfont\text{conv}}

\newcommand{\Supp}{\normalfont\text{Supp}}

\newcommand{\Hom}{\normalfont\text{Hom}}

\newcommand{\indeg}{\normalfont\text{indeg}}

\newcommand{\Hilb}{{\normalfont\text{Hilb}}}
\newcommand{\Spec}{\normalfont\text{Spec}}

\def\f0{\mathbf{0}}

\def\1{\mathbf{1}}

\newtheorem{theorem}{Theorem}[section]

\newaliascnt{headcor}{headthm}

\aliascntresetthe{headcor}

\newaliascnt{headconj}{headthm}

\aliascntresetthe{headconj}

\newaliascnt{corollary}{theorem}
\newtheorem{corollary}[corollary]{Corollary}
\aliascntresetthe{corollary}

\newaliascnt{claim}{theorem}

\aliascntresetthe{claim}

\newaliascnt{lemma}{theorem}
\newtheorem{lemma}[lemma]{Lemma}
\aliascntresetthe{lemma}

\newaliascnt{conjecture}{theorem}

\aliascntresetthe{conjecture}

\newaliascnt{proposition}{theorem}
\newtheorem{proposition}[proposition]{Proposition}
\aliascntresetthe{proposition}

\theoremstyle{definition}
\newaliascnt{definition}{theorem}
\newtheorem{definition}[definition]{Definition}
\aliascntresetthe{definition}

\newaliascnt{notation}{theorem}

\aliascntresetthe{notation}

\newaliascnt{example}{theorem}
\newtheorem{example}[example]{Example}
\aliascntresetthe{example}

\newaliascnt{examples}{theorem}

\aliascntresetthe{examples}

\newaliascnt{remark}{theorem}
\newtheorem{remark}[remark]{Remark}
\aliascntresetthe{remark}

\newaliascnt{question}{theorem}
\newtheorem{question}[question]{Question}
\aliascntresetthe{question}

\newaliascnt{questions}{theorem}

\aliascntresetthe{questions}

\newaliascnt{problem}{theorem}

\aliascntresetthe{problem}

\newaliascnt{construction}{theorem}

\aliascntresetthe{construction}

\newaliascnt{setup}{theorem}
\newtheorem{setup}[setup]{Setup}
\aliascntresetthe{setup}

\newaliascnt{algorithm}{theorem}

\aliascntresetthe{algorithm}

\newaliascnt{observation}{theorem}

\aliascntresetthe{observation}

\newaliascnt{defprop}{theorem}

\aliascntresetthe{defprop}

\def\equationautorefname~#1\null{(#1)\null}
\def\sectionautorefname~#1\null{Section #1\null}
\def\subsectionautorefname~#1\null{\S #1\null}

\def\opn#1#2{\def#1{\operatorname{#2}}}
\opn\Cl{Cl} \opn\deg{deg} \opn\Stab{Stab} \opn\aff{aff} \opn\div{div}
\opn\cone{cone} \opn\End{End} \opn\mod{mod}  \opn\pdim{pdim} \opn\diag{diag} \opn\vert{vert} \opn\m{m} \opn\V{V}
\opn\Cone{Cone} \opn\Pyr{Pyr} \opn\max{max} \opn\min{min} \opn\int{int} \opn\rev{rev} \opn\ker{ker} \opn\lat{lat} \opn\pull{pull}
\opn\cok{coker} \opn\ant{ant}
\opn\inte{int}
\opn\tr{tr}
\opn\rt{rt}

\title[A linear generalization of the nearly Gorenstein property]{A linear generalization of the nearly Gorenstein property,\\
with applications to Veronese subalgebras}

\author{Sora Miyashita}
\address[Miyashita]
{Department of Pure And Applied Mathematics, Graduate School Of Information Science And Technology, Osaka University, Suita, Osaka 565-0871, Japan}
\email{u804642k@ecs.osaka-u.ac.jp}

\date{\today}

\begin{document}

	\maketitle

\begin{abstract}
We studies the nearly Gorenstein property for Veronese subalgebras of (semi-)standard graded algebras.
We introduce a condition~$(\natural)$ for Cohen--Macaulay semi-standard graded rings, motivated by the study of Ehrhart rings.
We show that if a semi-standard graded algebra \( R \) satisfies~$(\natural)$, then its Veronese subalgebras \( R^{(k)} \) are nearly Gorenstein for all sufficiently large \( k \).
We also prove that if a standard graded algebra $R$ is nearly Gorenstein so does its Veronese subalgebra $R^{(k)}$ for all $k>0$.
\end{abstract}

\section{Introduction}
Throughout this introduction,
let $\kk$ be a field
and let \( R = \bigoplus_{n \ge 0} R_n \) be a Cohen--Macaulay graded ring with \( R_0 = \kk \).
Let \( \omega_R \) denote the canonical module of $R$
and let $\MM_R$ denote the graded maximal ideal of $R$.

In commutative algebra, the Gorenstein property plays a central role in the study of duality and the homological behavior of rings. To extend its influence to a broader range of rings, various generalizations have been developed. One such generalization is the notion of nearly Gorenstein rings, introduced by~\cite{herzog2019trace}.
The {\it canonical trace} \( \tr_R(\omega_R) \) is the sum of the images of all \( R \)-module homomorphisms from \( \omega_R \) to \( R \). The ring \( R \) is called \emph{nearly Gorenstein} if \( \tr_R(\omega_R) \supseteq \mm_R \).
It is known that \( R \) is {\it Gorenstein on the punctured spectrum} if and only if 
\( \sqrt{\operatorname{tr}_R(\omega_R)} \supseteq \mathfrak{m}_R \), where \( \sqrt{\phantom{a}} \) denotes the radical of an ideal.
The notions of nearly Gorenstein and being Gorenstein on the punctured spectrum have been studied for several prominent classes of rings, including generic determinantal rings (\cite{ficarra2024canonical!}), Hibi rings~(\cite[Theorem 3.4]{herzog2019measuring}), and Stanley--Reisner rings~(\cite{miyashita2024canonical}).
Moreover, for Ehrhart rings arising from lattice polytopes, an intermediate class between being Gorenstein on the punctured spectrum and being nearly Gorenstein has been studied.
To investigate the nearly Gorenstein property in this context, the \emph{floor polytope} $\lfloor P \rfloor$ and the \emph{remainder polytope} $\{P\}$ were introduced in~\cite{hall2023nearly} for a lattice polytope $P$ (we use a slightly modified definition of $\lfloor P \rfloor$).
When the identity \( P = \lfloor P \rfloor + \{ P \} \) holds, it gives rise to a polyhedral condition that lies strictly between the nearly Gorenstein property and the Gorenstein property on the punctured spectrum.
Here, the sum \( \lfloor P \rfloor + \{ P \} \) denotes the Minkowski sum of $\lfloor P \rfloor$ and $\{ P \}$.
It is known that for such a lattice polytope $P$, 
the Ehrhart ring $A(nP)$ becomes nearly Gorenstein for sufficiently large \( n \in \ZZ \)~(see \cite[Theorem 20]{hall2023nearly}).

The aim of this paper is to study a ring-theoretic property that generalizes the above property of Ehrhart rings, and to prove \cite[Theorem~20]{hall2023nearly} in a more general framework.
To set the stage for our approach, we begin by reviewing the notion of semi-standard graded rings.
We say that \( R \) is {\it standard graded} if it is generated as a \( \kk \)-algebra by \( R_1 \), and {\it semi-standard graded} if it is finitely generated as a \( \kk[R_1] \)-module.
For instance, Ehrhart rings are typically not standard graded, but they are always semi-standard graded. In addition to Ehrhart rings, face rings arising from simplicial posets~(\cite{stanley1991f})
are also known to be semi-standard graded.
Thus, several important classes of rings in combinatorial commutative algebra are typical examples of semi-standard graded rings.
We now introduce condition~\((\natural)\), which defines an intermediate class between nearly Gorenstein rings and those that are Gorenstein on the punctured spectrum.

\begin{definition}
Assume that $R$ is semi-standard graded.
We say that {\it \( R \) satisfies \((\natural)\)} if
$\sqrt{[\operatorname{tr}(\omega_R)]_1 R} \supseteq \mm_R,$
that is, the radical of the ideal generated by the degree-one part of the canonical trace contains \( \mm_R \).
\end{definition}

For Ehrhart rings $A(P)$ arising from lattice polytopes $P$, condition~$(\natural)$ is equivalent to the polyhedral condition $P=\lfloor P\rfloor+\{P\}$; that is, $A(P)$ satisfies $(\natural)$ if and only if $P=\lfloor P\rfloor+\{P\}$ (see~\autoref{thm:goodnews}).
It follows that condition~\((\natural)\) is not equivalent to the nearly Gorenstein property in general for Ehrhart rings
(see \cite[Example~19]{hall2023nearly}).
Moreover, we prove that condition $(\natural)$ implies a generalization of a result from \cite[Theorem 20]{hall2023nearly} in the setting of semi-standard graded domains, including Ehrhart rings as a special case~(see \autoref{THEMAINTHM}~(2)).

Throughout \autoref{THEMAINTHM} and \autoref{THEMAINTHM2}, we further assume that $R$ is semi-standard graded and that $\dim(R)>0$.  
In addition, for each integer \( k > 0 \), we denote the \( k \)-th Veronese subalgebra of \( R \) by \( R^{(k)} \).  
The next result is the main theorem of the paper.

\begin{theorem}[{\autoref{GREAT} and \autoref{cor:SUPERNICE}}]\label{THEMAINTHM}
The following hold:
\begin{itemize}
\item[\rm (1)] Suppose that \( R \) is standard graded.
If $R$ satisfies $(\natural)$,
then \( R^{(k)} \) is nearly Gorenstein for all integers $k > a_A$, where
\( a_{A} \) is the \( a \)-invariant of \(A:=R / \left[ \operatorname{tr}_R(\omega_R) \right]_1 R \);
\item[\rm (2)]
Suppose that \( R \) is a domain~(it need not be standard graded).
If $R$ satisfies $(\natural)$,
then there exists an integer \( k_R > 0 \) such that \( R^{(k)} \) is nearly Gorenstein for all \( k \ge k_R \).
\end{itemize}
\end{theorem}

Let $a_R$ denote the {\it $a$-invariant} of $R$.
Moreover, let \( \rt_{\kk[R_1]}(R) \) and \( \rt_{\kk[R_1]}(\omega_R) \) denote the maximal degrees of the minimal \( \kk[R_1] \)-generating set of \( R \) and \( \omega_R \), respectively.
While \autoref{THEMAINTHM} concerns condition~\((\natural)\), it yields the following result.

\begin{theorem}[{\autoref{cor:GREAT} and \autoref{prop:trace_stability}}]\label{THEMAINTHM2}
The following hold:
\begin{itemize}
\item[\rm (1)]
Suppose that $R$ is standard graded.
If $R$ is nearly Gorenstein, then so is \( R^{(k)} \) for all \( k > 0 \);
\item[\rm (2)]
Suppose that $R$ is a domain.
If $R^{(i)}$ is nearly Gorenstein for some $$i \ge \max\{1, \;|a_R|, \;\rt_{\kk[R_1]}(R), \;\rt_{\kk[R_1]}(\omega_R)\},$$
then so is $R^{(k)}$ for all $k>i$.
\end{itemize}
\end{theorem}

\autoref{THEMAINTHM2}~(1) follows from \autoref{THEMAINTHM}~(1) and generalizes \cite[Corollary~4.7]{herzog2019trace}.
In principle, the assumptions in \autoref{THEMAINTHM} cannot be weakened in general (see \autoref{rem:vero}).
It is still unclear whether the assumption that $R$ is a domain, which appears in \autoref{THEMAINTHM}~(2) and \autoref{THEMAINTHM2}~(2), is truly necessary.
However, we have verified that both results remain valid if this assumption is replaced by the condition that $R$ is level, i.e., that every minimal generating set of $\omega_R$ as an $R$-module consists of elements of the same degree.
This is shown in \autoref{cor:SUPERNICE}.


\subsection*{Outline}
The outline of this paper is as follows:  
In \autoref{sect_A}, we establish foundational concepts and definitions essential for the discussions that follow. We review basic notions and properties of commutative rings
 and affine semigroup rings.
In \autoref{sect_B}, we introduce a condition, denoted~\((\natural)\), for Cohen--Macaulay semi-standard graded rings and 
we prove \autoref{THEMAINTHM} and \autoref{THEMAINTHM2}, exploring the connection between condition~\((\natural)\) and Veronese subalgebras.
In \autoref{sect_D}, we give a criterion for affine semigroup rings to satisfy condition~$(\natural)$.


\begin{setup}\label{setup1}
Throughout this paper, we denote the set of non-negative integers by \( \mathbb{N} \) and let $\kk$ be a field.
Let \( R = \bigoplus_{i \in \NN} R_i \) be a positively graded Noetherian ring, and let 
\( \mathfrak{m}_R := \bigoplus_{i \in \NN} R_{i+1} \) denote its unique graded maximal ideal. 
We always assume that $R_0=\kk$.
In this setting, \( R \) has a graded canonical module \( \omega_R \) 
(see~\cite[Definition (2.1.2)]{goto1978graded}). Let \( a_R \) denote the \emph{\( a \)-invariant} of \( R \), defined by
$a_R := - \min \{ j \in \mathbb{Z} : (\omega_R)_j \neq 0 \}$.
For a graded $R$-module \( M \) and an integer \( i \in \mathbb{Z} \), the degree \( i \) component of \( M \) is often denoted by \( [M]_i \) instead of \( M_i \).
Let \( \mathcal{S} \subseteq R \) be the multiplicative set of homogeneous non-zero divisors. 
The \emph{graded total quotient ring} \( Q(R) \) is defined as the localization
$Q(R) := \mathcal{S}^{-1}R$.
Then \( Q(R) \) inherits a natural \( \mathbb{Z} \)-grading, with graded components given by
$[Q(R)]_i := \left\{ \frac{r}{s} \in Q(R) : r, s \in R\text{\;are homogeneous},\ \deg(r) - \deg(s) = i \right\}$ for $i \in \mathbb{Z}$.
\end{setup}

\section{Preliminaries and auxiliary lemmas}
\label{sect_A}

The goal of this section is to prepare the required materials for the discussions of our main results.
Throughout this section, unless otherwise noted, we maintain \autoref{setup1} and fix an integer $d>0$.

\subsection{Traces ideals and some generalizations of the Gorenstein property}
In this subsection, we recall some basic definitions and facts about trace ideals.
We also recall several standard generalizations of the Gorenstein property.
Note that the following \autoref{trace} does not assume \autoref{setup1}.
We define the trace ideal in this way to ensure consistency with the graded case.

\begin{definition}\label{trace}
Let $R$ be an arbitrary Noetherian commutative ring.
For an $R$-module $M$, the sum of all images of homomorphisms $\phi \in \Hom_R(M,R)$ is called the {\it trace} of $M$:
\[
\tr_R(M):=\sum_{\phi \in \Hom_R(M,R)}\phi(M).
\] 
\end{definition}

\begin{remark}\label{rem:interestingfiniteness}
For a finitely generated graded $R$-module $M$,
we have
$\Hom_R(M,R)={}^*\Hom_R(M,R)$
(see \cite[Exercises~1.5.19 (f)]{bruns1998cohen})
where ${}^*\Hom_R(M,R)$ denotes the set of graded homomorphisms from $M$ to $R$.
It follows that
$\tr_R(M)=\sum_{\phi \in {}^*\Hom_R(M,R)}\phi(M)$.
For any graded $R$-module $N$,
note that \( \operatorname{tr}_R(M) = \operatorname{tr}_R(N) \) if there exists an isomorphism $\phi \in \Hom^*_R(M,N)$.
\end{remark}

\begin{definition}
An \( \mathbb{N}^d \)-graded ring \( R \) is a ring endowed with a direct sum decomposition
$R = \bigoplus_{\mathbf{a} \in \mathbb{N}^d} R_{\mathbf{a}}$
such that \( R_{\mathbf{a}} \cdot R_{\mathbf{b}} \subseteq R_{\mathbf{a} + \mathbf{b}} \) for all \( \mathbf{a}, \mathbf{b} \in \mathbb{N}^d \).
A \( \mathbb{Z}^d \)-graded \( R \)-module \( M \) is an \( R \)-module together with a direct sum decomposition
$M = \bigoplus_{\mathbf{a} \in \mathbb{Z}^d} M_{\mathbf{a}}$
satisfying \( R_{\mathbf{b}} \cdot M_{\mathbf{a}} \subseteq M_{\mathbf{a} + \mathbf{b}} \) for all \( \mathbf{a} \in \mathbb{Z}^d \) and \( \mathbf{b} \in \mathbb{N}^d \).
\end{definition}

\begin{remark}\label{Mult}
Let \( R \) be an \( \mathbb{N}^d \)-graded ring and let \( M \) be a finitely generated \( \mathbb{Z}^d \)-graded \( R \)-module.
Then we have
$\operatorname{Hom}_R(M, R) = 
{}^*\operatorname{Hom}_R^{\ZZ^d}(M, R)$
where
${}^*\operatorname{Hom}_R^{\ZZ^d}(M, R)$
denotes the set of $\ZZ^d$-graded homomorphisms from $M$ to $R$.
It follows that $\operatorname{tr}_R(M) = \sum_{\varphi \in {}^*\operatorname{Hom}_R^{\ZZ^d}(M, R)} \varphi(M)$.
In particular,
\( \operatorname{tr}_R(M) \) is a homogeneous ideal of $R$ with respect to the \( \mathbb{N}^d \)-grading.
\end{remark}

\begin{definition}
Let $Q(R)$ denote the graded total quotient ring
(see \autoref{setup1}).
{\it A graded fractional ideal} of 
$R$ is a finitely generated graded
$R$-submodule $J \subseteq Q(R)$
that contains a {\it torsion-free element} $x \in J$ of $R$~(that is, $rx = 0$ implies $r = 0$ for all $r \in R$).
For a graded fractional ideal $J$ of $R$, we set
$J^{-1}=\{a \in Q(R) : aJ \subseteq R\}$.
\end{definition}

A useful formula for computing the trace of a fractional ideal is as follows:

\begin{lemma}\label{lemma:tracefractional}
Let $J \subseteq Q(R)$ be a graded fractional ideal.
Then we have $\tr_R(J)=J \cdot J^{-1}$.
\begin{proof}
Since $J$ is a graded fractional ideal,
$J$ is generated by $\left\{ \frac{b_1}{a_1}, \ldots, \frac{b_r}{a_r} \right\} \subseteq Q(R)$, where $a_i,b_i \in R$ and $a_1,\ldots,a_r$ is non-zero divisor of $R$.
Put $x=\prod_{i=1}^r a_i$.
Then $J$ is isomorphic to the graded ideal $xJ \subseteq R$ up to degree shift.
Since $xJ$ contains a homogeneous non-zero divisor of $R$,
we have
$\tr_R(J)=\tr_R(xJ)=(xJ) \cdot (xJ)^{-1}=J \cdot J^{-1}$ by \cite[Lemma 1.1]{herzog2019trace}.
\end{proof}
\end{lemma}



\begin{definition}[{see \cite[Definition 2.2]{herzog2019trace}}]
We say that $R$ is $\textit{nearly Gorenstein}$ if $\tr_R(\omega_R) \supseteq {\mm_R}$.
\end{definition}

\begin{remark}
$R$ is Gorenstein if and only if $\tr_R(\omega_R) = R$.
Moreover,
$R$ is Gorenstein on the punctured spectrum $\Spec(R) \setminus \{ {\mm_R} \}$ if and only if $\sqrt{\tr_R(\omega_R)} \supseteq {\mm_R}$. Therefore, if $R$ is nearly Gorenstein,
then $R$ is Gorenstein on the punctured spectrum.
\begin{proof}
It follows from the graded version of \cite[Lemma 2.1]{herzog2019trace}.
\end{proof}
\end{remark}

\begin{definition}[{see \cite[Proposition 3.3.18]{bruns1998cohen}}]
Suppose that $R$ is Cohen--Macaulay.
We say that $R$ is {\it generically Gorenstein} if  $R_\pp$ is Gorenstein for all minimal prime ideals $\pp$.
We say that $R$ is {\it level} if
$\omega_R$ is generated in a single degree.
\end{definition}

In the remainder of this section, we collect several technical lemmas that will be used in the proofs of our main results.
They concern the structure of trace ideals and the behavior of graded fractional ideals.

\begin{lemma}\label{NZDLOVE}
Suppose that $R_0 = \kk$ is an infinite field. If $f, g \in R$ and $f$ is a non-zero divisor, then there exists $r \in R_0 \setminus \{0\}$ such that $f + rg$ is a non-zero divisor in $R$.
\begin{proof}
Assume the contrary: for every $r \in R_0 \setminus \{0\}$, $f + rg$ is a zero divisor. Since $R$ is a Noetherian ring, the set of zero divisors is the union of a finite number of associated prime ideals $X := \{\pp_1, \pp_2, \ldots, \pp_n\}$ of $R$ (see \cite[Theorem 6.1~(ii)]{matsumura1989commutative}). Therefore, for each $r \in R_0 \setminus \{0\}$, there exists a prime ideal $\pp \in X$ such that $f + rg \in \pp$. Since $R_0$ is an infinite field, we can choose $r_1, r_2 \in R_0 \setminus \{0\}$ and $\pp \in X$ such that $f + r_ig \in \pp$ for $i = 1,2$. Then $r_i^{-1} f + g \in \pp$ for $i = 1,2$. Subtracting these, we obtain $(r_1^{-1} - r_2^{-1}) f \in \pp$. Since $r_1^{-1} \neq r_2^{-1}$, it follows that $f \in \pp$. This contradicts the assumption that $f$ is a non-zero divisor.
\end{proof}
\end{lemma}

\begin{definition}
For a non-zero graded $R$-module $M$, we set $\indeg_R(M) := \min\{i \in \ZZ : M_i \ne 0\}$.
\end{definition}

\begin{lemma}\label{cool}
Let $J \subseteq Q(R)$ be a graded fractional ideal, and suppose that there exists a torsion-free element $f \in J_{\indeg_R(J)}$ of $R$.
\begin{itemize}
\item[\rm (1)] Then a homogeneous generating set of $J$ as an $R$-module can be chosen to include $f$, with all elements being $R$-regular.
\item[\rm (2)] Let $\left\{ \frac{a_1}{b_1}, \frac{a_2}{b_2}, \ldots, \frac{a_l}{b_l} \right\} \subseteq Q(R)$ be a homogeneous minimal generating set of $J$ as an $R$-module, where each $a_i,b_i \in R$ are homogeneous non-zero divisors for all $1 \le i \le l$. Set $G(J) := \left\{ \frac{a_1}{b_1}, \ldots, \frac{a_l}{b_l} \right\}$ and $G_{\min}(J) := \{ f \in G(J) : \deg f = \indeg_R(J) \}$. If $h \in [\tr_R(J)]_{\indeg_R(\tr_R(J))}$, then we can write $h = \sum_{i=1}^k f_i g_i$, where $f_i \in G_{\min}(J)$ and $g_i \in [J^{-1}]_{\indeg_R(\tr_R(J)) - \indeg_R(J)}$ for each $1 \le i \le k$.
\end{itemize}
\begin{proof}
(1):
We may assume that $R_0 = \kk$ is an infinite field.
It is easy to choose a homogeneous generating set of $J$ that includes $f$. Furthermore, we may assume $J$ is an ideal. Let $f_1 = f, f_2, \ldots, f_l$ be a homogeneous generating set of $J$. We proceed by induction on $l$. The case $l = 1$ is trivial. Assume the statement holds for $l = k$. We show it also holds for $l = k+1$. Applying the induction hypothesis to $J' := (f, f_2, \ldots, f_k)R$, we can assume $f_2, \ldots, f_k$ are non-zero divisors. Since $R$ is semi-standard graded, by \autoref{rem:buruburu}~(3), there exists a non-zero divisor $z \in R_1$. Since $\deg f_{k+1} \ge \deg f$, by \autoref{NZDLOVE}, there exists $r \in R_0 \setminus \{0\}$ such that $F := fz^{\deg f_{k+1} - \indeg_R(J)} + rf_{k+1}$ is a homogeneous non-zero divisor. Then $\{f, f_2, \ldots, f_k, F\}$ is a generating set of $J = J' + (f_{k+1})R$, consisting entirely of non-zero divisors.

(2): Write $G(J) = \{f_1, \ldots, f_l\}$ and $G_{\min}(J) = \{f_1, \ldots, f_k\}$ with $1 \le k \le l$. Since $h \in J \cdot J^{-1}$ by \autoref{lemma:tracefractional}, we can write $h = \sum_{i=1}^l f_i g_i$, where $g_i \in J^{-1}$. It suffices to show $g_i = 0$ for any $f_i \notin G_{\min}(J)$. Suppose $g_i \ne 0$ for some such $f_i$. Note that $\deg f_i g_i = \deg h = \indeg_R(\tr_R(J))$. Since $g_i \in J^{-1}$, we have $f_1 g_i \in R$. But $\deg f_1 g_i = \deg f_1 + \deg g_i < \deg f_i + \deg g_i = \deg h$, which is impossible unless $f_1 g_i = 0$. Since $g_i$ is $R$-regular, it follows that $f_1 = 0$, a contradiction.
\end{proof}
\end{lemma}

\subsection{Semi-standard graded rings and Veronese subalgebras}
In this subsection, we recall the definitions of semi-standard graded rings and Veronese subalgebras, and briefly discuss their interaction.
In particular,
we prove \autoref{cont}, which is a key ingredient in the proof of the main results of this paper.


\begin{definition}
Let $R=\bigoplus_{i \in \NN} R_i$ be a Noetherian graded $\kk$-algebra over a field $R_0=\kk$.
If $R =\kk[R_1]$, that is, $R$ is generated by $R_1$ as a $\kk$-algebra, then we say that $R$ is {\it standard graded}.  
If $R$ is finitely generated as a $\kk[R_1]$-module, then we say that $R$ is {\it semi-standard graded}.
\end{definition}

If $R$ is semi-standard graded, its Hilbert series $\Hilb(R,t) := \sum_{i=0}^\infty \dim_\kk(R_i)t^i$ is of the following form:
\[
\Hilb(R,t) = \frac{h_0 + h_1 t + \cdots + h_s t^s}{(1-t)^{\dim R}},
\]
where $h_i \in \mathbb{Z}$ for all $0 \le i \le s$, $h_s \ne 0$, and $\sum_{i=0}^s h_i \ne 0$.
The sequence $(h_0,h_1,\ldots,h_s)$ is called the \emph{$h$-vector} of $R$, with $h_0 = 1$.
If $R$ is Cohen--Macaulay, then $h_i \ge 0$ holds for all $i$ (see~\cite{stanley1991hilbert}); 
furthermore, if $R$ is standard graded, then $h_i > 0$ for all $i$ (see~\cite{stanleyhilbert}).
The integer $s$, called the \emph{socle degree} of $R$, is denoted by $s(R)$. 
For a Cohen--Macaulay ring $R$, it is known that the degree of $\Hilb(R,t)$ is equal to $a_R$.
The Ehrhart rings $A(P)$ of lattice polytopes $P$ form a typical class of semi-standard graded rings. 


\begin{definition}
Let \( k > 0 \) be a positive integer. For a positively graded ring \( R = \bigoplus_{i \in \NN} R_i \), the \textit{\( k \)-th Veronese subalgebra} is defined by
$R^{(k)} := \bigoplus_{i \in \NN} R_{ik}$.
Although this may sound tautological, we emphasize that the grading on \( R^{(k)} \) is given by \( [R^{(k)}]_i = R_{ik} \) for each \( i \in \NN \).  
Hence, \( R^{(k)} \) is standard graded if and only if \( R^{(k)} = \kk[R_k] \).  
Similarly, for a graded \( R \)-module \( M = \bigoplus_{i \in \ZZ} M_i \), we define its \textit{\( k \)-th Veronese submodule} by
$M^{(k)} := \bigoplus_{i \in \ZZ} M_{ik}$.
\end{definition}

\begin{remark}\label{rem:YES}
Suppose that $R$ is semi-standard graded.
Let \( J \) be a graded fractional ideal of \( R \), minimally generated as an \( R \)-module by \( \{ f_1, \dots, f_s \} \), and let \( \{ \theta_1, \dots, \theta_t \} \) be a minimal generating set of \( R \) as a \( \kk[R_1] \)-module.  
Then \( \{ f_i \theta_j : 1 \le i \le s,\, 1 \le j \le t \} \) generates \( J \) as a \( \kk[R_1] \)-module.
\end{remark}

\begin{definition}
Let $M$ be a graded $R$-module. Let $\rt_R(M)$ be the maximum degree of the minimal generating set of $M$.
\end{definition}

\begin{remark}\label{rem:buruburu}
Suppose that $R$ is semi-standard graded.
Let $M$ be a finitely generated graded $R$-module.
Then the following hold:
\begin{itemize}
\item[(1)] We have $R^{(n)} = \kk[R_n]$ for any $n \ge \rt_{\kk[R_1]}(R)$;
\item[(2)] We have $M_{n+k} = R_k M_n$ for all $n \ge \rt_{\kk[R_1]}(M)$ and $k \ge 0$;
\item[(3)]
If $\depth(R)> 0$,
then $R_1$ contains a non-zero divisor.
\end{itemize}
\end{remark}
\begin{proof}
(1) follows from \cite[Chapter III (*), Section 3, Lemma 2~(ii)]{bourbaki1998commutative}).

(2): For $M = R$, this is a special case of \cite[Chapter III (*), Section 3, Lemma 2~(i)]{bourbaki1998commutative}. For an arbitrary $R$-module $M$, we use induction on $k$. For $k = 1$, (2) follows from \cite[Chapter III (*), Section 3, Lemma 1]{bourbaki1998commutative}. Assume it holds for $k$ and prove it for $k + 1$. By the induction hypothesis and noting it holds for $M = R$, we have
$M_{n+(k+1)} = M_{(n+k)+1} = R_1 M_{n+k} = R_1 R_k M_n = R_{k+1} M_n$
Thus, it holds for $k + 1$.

(3):
Let $A = R \otimes_\kk \kk(x)$,
where \( \kk(x) \) denote the field of rational functions over \( \kk \) in one variable.
Notice that $1=\indeg_A(\mm_A)=\indeg_R(\mm_R)$.
By \cite[Proposition 1.5.12]{bruns1998cohen}, we can take a non-zero divisor $f = r \otimes 1_\kk \in A_1$, where $r \in R_1$. We can regard $R$ as a subalgebra of $A$ via the injection $\iota: R \rightarrow A$ defined by $\iota(x) = x \otimes 1_\kk$ for any $x \in R$. Thus, $r$ is a non-zero divisor in $R_1$ because $\iota(r) = f$ is a non-zero divisor.
\end{proof}

\subsection{Affine semigroup rings}
To present several important examples in this paper, we briefly recall some basic definitions of affine semigroup rings.
Throughout, we fix an integer $d > 0$ and a field $\kk$.

An \textit{affine semigroup} $S$ is a finitely generated subsemigroup of $\mathbb{Z}^d$. For $X \subseteq S$, we denote by $\langle X \rangle$ the subsemigroup of $S$ generated by $X$. We denote the group generated by $S$ by $\ZZ S$.
We say that $S$ is \emph{pointed} if $S \cap (-S) = \{0\}$.
Let $A=\kk[x_1, \ldots, x_d]$ denote the polynomial ring in $d$ variables.
For a vector $\bfa = (a_1, \ldots, a_d) \in \NN^d$, we often use the notation 
$\bfx^{\bfa} := \prod_{i=1}^d x_i^{a_i} \in A$.

\begin{definition}
Let $S$ be an affine semigroup. The \emph{affine semigroup ring} \( \kk[S] \) is defined as the \( \kk \)-vector space with basis \( \{ \bfx^\bfa \mid \bfa \in S \} \), endowed with multiplication $\bfx^\bfa \cdot \bfx^\bfb = \bfx^{\bfa + \bfb}$ for all $\bfa, \bfb \in S$.
\end{definition}

If $S$ is pointed, then $\kk[S]$ is isomorphic to a positively graded affine semigroup ring (see~\cite[Proposition~6.1.5]{bruns1998cohen}).
Moreover, every pointed affine semigroup $S$ has a unique finite minimal generating set (see~\cite[Proposition~7.15]{miller2005combinatorial}).
Let $S$ be a pointed affine semigroup, and let $G_S = \{ \bfa_1, \ldots, \bfa_s \} \subseteq \NN^d$ be its minimal generating set. Fix the affine semigroup ring $R = \kk[S]$. Recall that $R$ is naturally $\NN^d$-graded, and that every nonzero element $f \in R$ can be uniquely written as a sum of its homogeneous components $f_1, f_2, \ldots, f_n \in R$. We define the \emph{support} of $f$ by $\Supp(f) := \{ f_1, f_2, \ldots, f_n \}$.
A subset \(F \subseteq S\) is called a \emph{face} of \(S\) if, for all \(a,b \in S\),
$a+b \in F \Leftrightarrow a \in F \text{ and } b \in F$.
We define the \emph{dimension} of a face \( F \subseteq S \) by $\dim F := \operatorname{rank}_{\ZZ}(\ZZ F)$, where \( \ZZ F \) denotes the subgroup of \( \ZZ^d \) generated by \( F \). We set $E_S := \{ \bfa \in G_S : \dim(\NN \bfa) = 1 \}$.
With this terminology in place, we prove the following lemma.

\begin{lemma}\label{para}
Let $S \subseteq \ZZ^d$ be a pointed affine semigroup,
where
$d=\dim(R)$.
Let $R = \kk[S]$ be a Cohen--Macaulay semi-standard graded affine semigroup ring, and let $\theta_1, \ldots, \theta_d \in R_1$ be a linear system of parameters. Then for any $\bfx^\bfe$ with $\bfe \in E_S$, there exists some $\theta_i$ such that $\bfx^\bfe \in \Supp(\theta_i)$.
\end{lemma}
\begin{proof}
Put $B=\kk[\theta_1,\ldots,\theta_d]$.
Since $R$ is a finitely generated graded $B$-module, there exist homogeneous monomials
$\bfx^{\bfa_1},\ldots,\bfx^{\bfa_r}\in R$ such that
$R=\sum_{i=1}^r B\,\bfx^{\bfa_i}$.
Choose $k\in\ZZ_{>0}$ such that $k>\deg(\bfx^{\bfa_i})$ for all $i$.
Note that $\deg(\bfx^\bfe)=1$ (see~\cite[Proposition~3.4]{miyashita2023comparing}).
Then $\bfx^{k\bfe}\in R_k$, so we can write
\[
\bfx^{k\bfe}=\sum_{i=1}^r g_i(\theta_1,\ldots,\theta_d)\,\bfx^{\bfa_i},
\]
where each $g_i$ is homogeneous in $B$ of degree $k-\deg(\bfx^{\bfa_i})$.
Expand each $g_i$ as a $\kk$-linear combination of monomials in the $\theta$'s, and expand each
$\theta_j\in R_1$ as a $\kk$-linear combination of monomials in $\Supp(\theta_j)$.
Since the monomials in $R=\kk[S]$ form a $\kk$-basis, the monomial $\bfx^{k\bfe}$ must occur
with nonzero coefficient in this expansion.
Hence there exist an index $i$ and monomials $y_1,\ldots,y_m$ such that
$m=k-\deg(\bfx^{\bfa_i})>0$, $y_j\in \Supp(\theta_{n_j})\subset R_1$ for some $n_j\in\{1,\ldots,d\}$,
and $y_1\cdots y_m\,\bfx^{\bfa_i}=\bfx^{k\bfe}$. Let $y_j=\bfx^{\bfb_j}$
and
let $F:=\NN\bfe\subseteq S$.
The equality of monomials above means an equality in $S$:
\[
k\bfe=\bfa_i+\sum_{i=1}^m\bfb_i.
\]
Since $F=\NN\bfe$ is a face, the above equality forces $\bfb_j\in F$ for every $j$.
Thus, for each $j$, there exists $n_j\in\NN$ such that $\bfb_j=n_j\bfe$.
Since $y_j\in R_1$ and $\deg(\bfx^\bfe)=1$, we have $n_j=1$, i.e., $\bfb_j=\bfe$ for all $j$.
Hence $\bfx^\bfe = y_{j_0} \in \Supp(\theta_{n_{j_0}})$ for some $j_0$,
as desired.
\end{proof}

\section{Rings satisfying $(\natural)$ and their Veronese subalgebras}
\label{sect_B}
In this section, we introduce condition $(\natural)$ and examine its relationship with Veronese subalgebras.
Throughout this section, as in the previous one, and unless otherwise noted, we maintain \autoref{setup1}.
As mentioned in the introduction, we now restate the definition of condition~$(\natural)$ introduced in this work.

\begin{definition}
Suppose that $R$ is Cohen--Macaulay and semi-standard graded.  
We say that \emph{$R$ satisfies $(\natural)$} if  
$\sqrt{[\tr_R(\omega_R)]_1 R} \supseteq \mm_R$.
\end{definition}

It is immediate from the definition that if $R$ is nearly Gorenstein, then it satisfies~$(\natural)$.  
Furthermore, if $R$ satisfies~$(\natural)$, then it is Gorenstein on the punctured spectrum.  
We provide the following concrete example of a ring that satisfies condition~$(\natural)$ but fails to be nearly Gorenstein by a narrow margin.

\begin{example}\label{ex_nice}
Let $A = \QQ[x_1,x_2,x_3,x_4,x_5]$ be a polynomial ring and let
\[
I = (x_4^2 - x_3x_5,\ x_3x_4 - x_2x_5,\ x_3^2 - x_2x_4,\ x_1^2,\ x_2^2x_4,\ x_2^2x_3,\ x_2^3)A
\]
be an ideal of $A$, and set $R = A/I$.
Note that $R$ is standard graded, where the grading is $\deg x_i = 1$ for each $1 \le i \le 5$.
Then we can check that $R$ is a 1-dimensional Cohen--Macaulay ring with  
$\tr_R(\omega_R) = (x_2,x_3,x_4,x_5)R (=[\tr_R(\omega_R)]_1R)$  
by using \texttt{Macaulay2}~\cite{M2} and \cite[Corollary 3.2]{herzog2019trace}.  
Thus, $R$ is not nearly Gorenstein. However, $R$ satisfies $(\natural)$ because we can check that  
$\sqrt{[\tr_R(\omega_R)]_1 R} = \sqrt{\tr_R(\omega_R)} = \mm_R$.
\end{example}

Below, we point out several other important facts.

\begin{remark}\label{rem:PUNC}
Suppose that $R$ is Cohen--Macaulay.
Let $\KK'$ be an extension field of $R_0 = \kk$ and set $A = R \otimes_\kk \KK'$.
The following assertions hold; in (3) and (4), we assume that $R$ is semi-standard graded.
\begin{itemize}
\item[\rm (1)] $\tr_{A}(\omega_A) = \tr_R(\omega_R)A$;
\item[\rm (2)] $R$ is nearly Gorenstein if and only if $A$ is nearly Gorenstein;
\item[\rm (3)] If there exists a graded parameter ideal
$I = (\theta_1,\ldots,\theta_{\dim(R)})$
with $\theta_1,\ldots,\theta_{\dim(R)} \in R_1$
such that $\tr_R(\omega_R) \supseteq I$,
then $R$ satisfies $(\natural)$.  
If $\kk$ is an infinite field, the converse is also true;
\item[\rm (4)] If $R$ satisfies $(\natural)$, then so does $A$.
\end{itemize}
\begin{proof}
(1) follows from \cite[Exercises 3.3.31]{bruns1998cohen} and \cite[Lemma 1.5 (iii)]{herzog2019trace}.
(4) follows from
(1).

(2):
If $R$ is nearly Gorenstein, then so is $A$ by (1).  
If $A$ is nearly Gorenstein, then $\tr_R(\omega_R)A \supseteq \mm_R A$ by (1).  
Thus $\tr_R(\omega_R)A \cap R \supseteq \mm_R A \cap R$.  
Since $R \to A$ is faithfully flat, we have  
$\tr_R(\omega_R)A \cap R = \tr_R(\omega_R)$  
and  
$\mm_R A \cap R = \mm_R$  
by \cite[Theorem 7.5 (ii)]{matsumura1989commutative}.  
Therefore, $\tr_R(\omega_R) \supseteq \mm_R$, and hence $R$ is nearly Gorenstein.
(3):
The first part is clear. The second part follows from \cite[Proposition 1.5.12]{bruns1998cohen}.
\end{proof}
\end{remark}

From this point on, we prove the main results of this paper.
The following result applies not only to semi-standard graded rings but also to general Noetherian graded rings.


\begin{theorem}\label{useful}
Let $\fkn = (R_1)R$. Let $I$ be a graded ideal generated by a subset of $R_1$, and let $J$ be a graded fractional ideal of $R$. Suppose that there exists a homogeneous non-zero divisor $y \in R_1$.
If
$\tr_{R}(J) \supseteq I$,
then we have $\tr_{R ^{(k)}}(J ^{(k)}) \supseteq (\fkn^{k-1}I) ^{(k)}$ for any $k>0$.
\begin{proof}
Let $\{\theta_1,\ldots,\theta_s\} \subseteq R_1$ and $\{x_1,\ldots,x_t\}$ be minimal generating sets of $I$ and $\fkn$, respectively, as $R$-modules.
Throughout this proof, we set $\left( \prod_{\emptyset} x_{t_h} \right) := 1$ by convention.
Then we have
\[
\fkn^{k-1}I = \left( \left( \prod_{h=1}^{k-1} x_{t_h} \right) \theta_i \;:\; 1 \le t_1 \le \cdots \le t_{k-1} \le t,\ 1 \le i \le s \right) R,
\]
It follows that
\[
(\fkn^{k-1}I)^{(k)} = \left( \left( \prod_{h=1}^{k-1} x_{t_h} \right) \theta_i \;:\; 1 \le t_1 \le \cdots \le t_{k-1} \le t,\ 1 \le i \le s \right) R^{(k)}.
\]
Therefore, it suffices to show that
$\left( \prod_{h=1}^{k-1} x_{t_h} \right) \theta_i \in J^{(k)} \cdot (J^{(k)})^{-1} = \tr_{R^{(k)}}(J^{(k)})$
for all $1 \le i \le s$ and $1 \le t_1 \le \cdots \le t_{k-1} \le t$ (see \autoref{lemma:tracefractional}).
Let $G(J) = \{f_1, \ldots, f_u\} \subseteq Q(R)$ be a minimal generating set of $J$ as an $R$-module. Note that $\tr_R(J) = J \cdot J^{-1}$ by \autoref{lemma:tracefractional}. Hence, for each $1 \le i \le s$, since $\theta_i \in I \subseteq \tr_R(J)$, there exist homogeneous elements $g_{i,j} \in J^{-1}$ such that
$\theta_i = \sum_{j=1}^u f_j g_{i,j}$.

For each \( 1 \le j \le u \), either \( g_{i,j} = 0 \) or \( \deg g_{i,j} = 1 - \deg f_j \) holds, as \( \deg \theta_i = 1 \) by \autoref{cool}~(2).
If $g_{i,j} \neq 0$, then we may write
$g_{i,j} = \frac{b_{i,j}}{a_{i,j}}$,
where $a_{i,j}$ and $b_{i,j}$ are homogeneous elements of $R$, with $a_{i,j}$ a non-zero divisor in $R$, and
$\deg b_{i,j} = \deg a_{i,j} + 1 - \deg f_j$.
Furthermore, we can write $\deg f_j = k d_j - r_j$ for some integer $d_j > 0$ and $0 \le r_j < k$. Define
$l_{i,j} := \min \left\{ l \in \mathbb{N} \;\middle|\; l + \deg a_{i,j} \in k\mathbb{Z} \right\}$,
and set
$p_j := f_j \prod_{h=1}^{r_j} x_{t_h}$
and
$q_{i,j} := \frac{ y^{l_{i,j}} b_{i,j} \prod_{r_j < h \le k-1} x_{t_h} }{ y^{l_{i,j}} a_{i,j} }$.
Then we have $p_j \in J^{(k)}$ and $q_{i,j} \in (J^{(k)})^{-1}$.

Indeed, since $p_j \in J$ and
$\deg(p_j) = r_j + \deg f_j = k d_j$,
we obtain $p_j \in J^{(k)}$. Moreover, since
$\deg \left( y^{l_{i,j}} a_{i,j} \right) = l_{i,j} + \deg a_{i,j} \in k\mathbb{Z}$
and
$\deg \left( y^{l_{i,j}} b_{i,j} \prod_{r_j < h \le k-1} x_{t_h} \right) = (l_{i,j} + \deg a_{i,j}) + k(1 - d_j) \in k\mathbb{Z}$,
we have $q_{i,j} \in R^{(k)}$. Since $q_{i,j} \in J^{-1}$, it follows that $q_{i,j} \in (J^{(k)})^{-1}$.

Therefore, we have
\[
\left( \prod_{h=1}^{k-1} x_{t_h} \right) \theta_i
= \left( \prod_{h=1}^{k-1} x_{t_h} \right) \sum_{j=1}^u f_j g_{i,j}
= \sum_{j=1}^u
\left(
\left( \prod_{h=1}^{k-1} x_{t_h}\right) f_j  \cdot \frac{b_{i,j}}{a_{i,j}} \right)
= \sum_{j=1}^u p_jq_{i,j} \in J^{(k)} \cdot (J^{(k)})^{-1}.
\]
\end{proof}
\end{theorem}

\begin{proposition}\label{Yeah}
Suppose that $R$ is a Cohen--Macaulay semi-standard graded ring with $\dim(R) > 0$.
Let $I$ be a graded ${\mm_R}$-primary ideal 
generated by a subset of $R_1$ and put $\fkn=(R_1)R$.
If $\tr_R(\omega_R) \supseteq I$, then we have $\tr_{R ^{(k)}}(\omega_{R ^{(k)}}) \supseteq (\fkn^{k-1}I) ^{(k)}$ for any $k > 0$.
\begin{proof}
Notice that we can take a non-zero divisor of $y \in R_1$ by \autoref{rem:buruburu}~(3) because $R$ is semi-standard graded.
Since $\sqrt{\tr_{R}(\omega_{R})} \supseteq \mm_R$,
$R$ is
generically Gorenstein because $\dim(R)> 0$. Thus, $\omega_R$ is isomorphic to a graded fractional ideal of $R$. For any $k > 0$, since $\omega_{R ^{(k)}} \cong (\omega_{R}) ^{(k)}$ (see \cite[Corollary (3.1.3)]{goto1978graded}) and by \autoref{useful}, we have $\tr_{R ^{(k)}}(\omega_{R ^{(k)}}) \supseteq (\fkn^{k-1}I) ^{(k)}$.
\end{proof}
\end{proposition}

We now present one of the main results of this section.

\begin{theorem}\label{GREAT}
Suppose that $R$ is a Cohen--Macaulay semi-standard graded ring.
Then the following hold:
\begin{itemize}
\item[\rm (1)] If $R$ satisfies $(\natural)$,
then so does $R^{(k)}$ for any $k>0$;
\item[\rm (2)] Suppose that $R$ is standard graded and $\dim(R)> 0$.
If $R$ satisfies $(\natural)$, then $R ^{(k)}$ is nearly Gorenstein for any $k > a_{R/[\tr_R(\omega_R)]_1 R}$.
\end{itemize}
\begin{proof}
(1):
It is clear that $R$ satisfies $(\natural)$ if $\dim(R)=0$,
so we may assume that $\dim(R)>0$.
Put $\fkn=(R_1)R$.
Take any $k>0$.
By \autoref{Yeah},
we have
$\tr_{R ^{(k)}}(\omega_{R ^{(k)}}) \supseteq (\fkn^{k-1}I) ^{(k)}$,
where $I=[\tr_R(\omega_R)]_1R$.
Thus $\fkn^{k-1}I$ is an ${\mm_R}$-primary ideal since both of $I$ and $\fkn^{k-1}$ are ${\mm_R}$-primary ideal.
Thus $(\fkn^{k-1}I) ^{(k)}$ is an ${\mm} ^{(k)}$-primary ideal of $R ^{(k)}$ generated by a subset of $[R ^{(k)}]_1$,
so $R ^{(k)}$ satisfies $(\natural)$.

(2):
Put $I=[\tr_R(\omega_R)]_1R$.
Notice that
$a_{R/I}=s_{R/I}$ and $I \supseteq \mm_R^{s(R/I)+1}$.
Take any $k > s(R/I)$.
Since $R$ is standard graded and $I$ is a ${\mm_R}$-primary ideal generated by subset of $R_1$, we have
$$\tr_{R ^{(k)}}(\omega_{R ^{(k)}}) \supseteq (\mm_R^{k-1}I) ^{(k)}=I ^{(k)}
 \supseteq (\mm_R^{s(R/I)+1}) ^{(k)}={\mm} ^{(k)}$$
by \autoref{Yeah},
so $R ^{(k)}$ is nearly Gorenstein.
\end{proof}
\end{theorem}

\autoref{GREAT}~(2) is a refinement of \cite[Corollary~4.7]{herzog2019trace}.  
Indeed, the following generalization of \cite[Corollary~4.7]{herzog2019trace} follows from \autoref{GREAT}~(2).

\begin{corollary}\label{cor:GREAT}
Suppose that $R$ is a Cohen--Macaulay standard graded ring with $\dim(R) > 0$.
If $R$ is nearly Gorenstein, then so is $R ^{(k)}$ for any $k > 0$.
\begin{proof}
It follows from \autoref{GREAT}~(2)
because $a_{R/[\tr_R(\omega_R)]_1R}=a_\kk=0$.
\end{proof}
\end{corollary}

\begin{example}\label{ex_good}
Let us consider \autoref{ex_nice} again.
Recall that $R$ is not nearly Gorenstein, but does satisfy $(\natural)$.
Moreover,
we can check that $[\tr_R(\omega_R)]_1R \supseteq \mm^2$
by using $\mathtt{Macaulay2}$ ({\cite{M2}}),
so we have $a_{R/[\tr_R(\omega_R)]_1R}=1$.
Thus
$R ^{(k)}$ is nearly Gorenstein for any $k>1$
by \autoref{GREAT}~(2).
\end{example}

\begin{remark}\label{rem:vero}
There are some remarks about \autoref{GREAT}.
\begin{itemize}
\item[\rm (1)]
\autoref{GREAT}~(2) does not hold in general if we replace condition~$(\natural)$ with the property of being Gorenstein on the punctured spectrum.  
Indeed, there exists a Cohen--Macaulay standard graded affine semigroup ring $A$ that is Gorenstein on the punctured spectrum, but such that $A^{(k)}$ is not nearly Gorenstein for any $k > 0$.
\item[\rm (2)]
\autoref{GREAT}~(2) does not hold in general for semi-standard graded rings.  
Indeed, there exists a nearly Gorenstein semi-standard graded affine semigroup ring $B$ such that $B^{(2)}$ is not nearly Gorenstein.
\end{itemize}
\end{remark}
\begin{proof}
(1): Set \( S = \langle (0,0,1),(2,2,3),(4,2,3),(3,3,4),(4,3,4) \rangle \), and let \( A = \QQ[S] \).
Then \( A \) is a standard graded affine semigroup ring, where
$\deg u =
\deg s^2t^2u^3 =
\deg s^4t^2u^3 =
\deg s^3t^3u^4 =
\deg s^4t^3u^4 = 1$,
since each generator of \( S \) lies on the affine hyperplane defined by \( y = z - 1 \), which does not contain the origin.
Using \texttt{Macaulay2}~\cite{M2}, one can verify that \( A \) is a 3-dimensional Cohen--Macaulay standard graded affine semigroup ring with \( a_A = 0 \) and \( h_{s(A)} = 1 \).
Moreover, by \cite[Corollary~3.2]{herzog2019trace} and \texttt{Macaulay2}, we obtain
$\tr_A(\omega_A) = (u, s^2t^2u^3, s^4t^2u^3, (s^3t^3u^4)^2, (s^4t^3u^4)^2)A$.
Thus, \( A \) is not Gorenstein but is Gorenstein on the punctured spectrum, since \( \sqrt{\tr_A(\omega_A)} = \mm_A \).

By \cite[Corollary~3.1.3 and Theorem~3.2.1]{goto1978graded}, for any \( k > 0 \), the Veronese subalgebra \( A^{(k)} \) is a non-Gorenstein Cohen--Macaulay standard graded affine semigroup ring with \( h_{s(A^{(k)})} = 1 \). Therefore, by \cite[Corollary~3.12~(1)]{miyashita2025pseudo}, \( A^{(k)} \) is not nearly Gorenstein.

(2):
Set \( S' = \langle (0,1), (3,1), (6,1), (9,1), (1,2), (4,2) \rangle \), and let \( B = \QQ[S'] \).
Then \( B \) is a nearly Gorenstein semi-standard graded affine semigroup ring with grading given by
\[
\deg t = \deg s^3t = \deg s^6t = \deg s^9t = 1, \quad \deg st^2 = \deg s^4t^2 = 2
\]
(see \cite[Example~3.8~(a)]{miyashita2023comparing}).
However, the Veronese subalgebra
$B^{(2)} \cong \QQ[t, st, s^3t, s^4t, s^6t, s^9t, s^{12}t, s^{15}t, s^{18}t]$
is not nearly Gorenstein.
Indeed, by \cite[Corollary~3.2]{herzog2019trace} and \texttt{Macaulay2}~\cite{M2}, we have
\[
\tr_{B^{(2)}}(\omega_{B^{(2)}}) = (t, s^3t, s^6t, s^9t, s^{12}t, s^{15}t, s^{18}t)B^{(2)} + (st, s^4t)^2 B^{(2)}.
\]
Hence \( B^{(2)} \) is not nearly Gorenstein, as \( st \notin \tr_{B^{(2)}}(\omega_{B^{(2)}}) \).
\end{proof}

\begin{remark}
Suppose that $R$ is a Cohen--Macaulay semi-standard graded ring with $\dim(R) > 0$.
If $R$ satisfies $(\natural)$, then we can find an integer $K>0$ such that $R ^{(K)}$ is nearly Gorenstein.
\begin{proof}
Note that there exists $r > 0$ such that $R ^{(r)} = \kk[R_{r}]$ (e.g., $r = \rt_{\kk[R_1]}(R)$). Since $R ^{(r)}$ is standard graded, by \autoref{GREAT}~(2), $R ^{(kr)} = (R ^{(r)}) ^{(k)}$ is nearly Gorenstein for any $k > a_{R^{(r)}/\tr_{R ^{(r)}}(\omega_{R ^{(r)}})}$. Therefore, if we take $K = (a_{R ^{(r)}/\tr_{R ^{(r)}}(\omega_{R ^{(r)}})} + 1)r$, then $R ^{(K)}$ is nearly Gorenstein.
\end{proof}
\end{remark}

\begin{question}\label{Q:tion}
Suppose that $R$ is a Cohen--Macaulay semi-standard graded ring with $\dim(R) > 0$.
If $R$ satisfies $(\natural)$,
then can we find an integer $K>0$ such that $R^{(k)}$ nearly Gorenstein for any $k \ge K$?
\end{question}

In the remainder of this section, we prove that \autoref{Q:tion} holds when \(R\) is either a domain or a level ring.
The following is needed to resolve \autoref{Q:tion}; its proof is somewhat intricate but important.

\begin{proposition}\label{cont}
Suppose that $R$ is semi-standard graded.
Let \( J \) be a graded fractional ideal of \( R \), and let
$K \ge \max \left\{1, \;|\indeg_R(J)|, \;\rt_{\kk[R_1]}(J), \;\rt_{\kk[R_1]}(R)\right\}$
be an integer.
Assume that $J_{\indeg_R(J)}$ contains a torsion-free element of $R$.
If \( \tr_{R^{(K)}}(J^{(K)}) \supseteq \MM_{R^{(K)}} \), then we have
$\tr_{R^{(i)}}(J^{(i)}) \supseteq \MM_{R^{(i)}}$ for all \( i \ge K \).
\end{proposition}
\begin{proof}
We proceed by induction on \( i \ge K \). The case \( i = K \) holds by assumption.  
suppose that \( \tr_{R^{(i)}}(J^{(i)}) \supseteq \MM_{R^{(i)}} \) for some \( i \ge K \).  
Since \( i \ge \rt_{\kk[R_1]}(R) \), both \( R^{(i)} \) and \( R^{(i+1)} \) are standard graded.
Moreover, by \autoref{cool}~(1), we may assume that all generators of \( J \) are \( R \)-regular elements.
To prove the claim, it suffices to show that any \( y \in R_{i+1}\) belongs to \( \tr_{R^{(i+1)}}(J^{(i+1)}) \).
Since $i \ge \rt_{\kk[R_1]}(R)$ and \autoref{rem:buruburu}~(2),
we have
\( R_{i+1} = R_1 R_i \).
Thus
we can write \(y=\sum_{j} x_j y_j'\) for some \(x_j\in R_1\) and \(y_j'\in R_i\).
Since \(\operatorname{tr}_{R^{(i+1)}}(J^{(i+1)})\) is an ideal, to prove
\(y\in \operatorname{tr}_{R^{(i+1)}}(J^{(i+1)})\) it is enough to show that
\(x_jy_j'\in \operatorname{tr}_{R^{(i+1)}}(J^{(i+1)})\) for all \(j\).
Fix \(j\). For notational simplicity, set \(x:=x_j \in R_1\) and \(y':=y_j' \in R_i\), and it remains to show that
\(xy'\in \operatorname{tr}_{R^{(i+1)}}(J^{(i+1)})\).

Since \(R\) is semi-standard graded, note that \(J_j \neq 0\) for all \(j \ge \indeg_R(J)\) by \autoref{rem:buruburu}~(3) and the assumption that \(J_{\indeg_R(J)} \neq 0\).
Moreover, since \(i \ge |\indeg_R(J)|\), we have \(\indeg_{R^{(i)}}(J^{(i)})=1\) if \(\indeg_R(J)\ge 0\), and \(\indeg_{R^{(i)}}(J^{(i)})=0\) if \(\indeg_R(J)<0\).
From here on, we show that \(xy' \in \tr_{R^{(i+1)}}(J^{(i+1)})\) by a case analysis depending on the values of \(\indeg_R(J)\) and \(\rt_{\kk[R_1]}(J)\).

\begin{itemize}
\item[\rm (1)] Assume that \(\indeg_R(J)\ge 0\).
In this case, for \(\alpha\in\{i,i+1\}\) we have \(J^{(\alpha)}=R^{(\alpha)}J_\alpha\).
Indeed, fix \(n\ge 1\). Since \(n\alpha \ge \indeg_{R^{(\alpha)}}(J^{(\alpha)})=\alpha\), we have \((n-1)\alpha \ge 0\), and since \(\alpha \ge \rt_{\kk[R_1]}(J)\), applying \autoref{rem:buruburu}~(2) yields
$J_{n\alpha}=R_{(n-1)\alpha}J_\alpha=[R^{(\alpha)}]_{n-1}J_\alpha$.
Hence \(J^{(\alpha)}=R^{(\alpha)}J_\alpha\).
We further divide into two subcases:
\begin{itemize}
  \item[(a)] If \(\tr_{R^{(i)}}(J^{(i)})=\MM_{R^{(i)}}\bigl(=R^{(i)}\,[\MM_{R^{(i)}}]_1\bigr)\), then
$\indeg_{R^{(i)}}\!\bigl(\tr_{R^{(i)}}(J^{(i)})\bigr)=1$.
Hence
$y' \in R_i = [R^{(i)}]_1 \subseteq [\MM_{R^{(i)}}]_1
= \bigl[\tr_{R^{(i)}}(J^{(i)})\bigr]_1$.
Therefore, by \autoref{lemma:tracefractional} and \autoref{cool}~(2), we can write
$y'=\sum_{k=1}^m f_k g_k$
with $f_k\in [J^{(i)}]_1$ and $g_k\in \bigl[(J^{(i)})^{-1}\bigr]_0$.
  We claim that \( g_k \in (J^{(i+1)})^{-1} \)
  (equivalently, \( g_k J_{i+1} \subseteq R^{(i+1)}\) because $J^{(i+1)}=R^{(i+1)}J_{i+1}$)
  for each \( k \). To verify this, note that it suffices to show \( g_k J_{i+1} \subseteq R\) because $g_k J_{i+1} \subseteq J_{i+1}$.
Take any \(a\in J_{i+1}\).
Since \(J_{i+1}=R_1J_i\) by \autoref{rem:buruburu}~(2), we can write
\( a = \sum_{l} r_l b_l \) with \( r_l \in R_1 \) and \( b_l \in J_i \). Then we have
  $g_k a =
  \sum_{l} r_l (g_k b_l) \in R^{(i+1)},$
  so \( g_k J_{i+1} \subseteq R \), as claimed.
  Consequently, $xy'=\sum_{k} (xf_k) g_k \in J^{(i+1)} \cdot (J^{(i+1)})^{-1}$,
  as desired.
  \item[(b)] If \( \tr_{R^{(i)}}(J^{(i)}) = R^{(i)} \), then \( \indeg_{R^{(i)}}(\tr_{R^{(i)}}(J^{(i)})) = 0 \), and we can write
  $1 = \sum_{h} f_h g_h$ with $f_h \in [J^{(i)}]_1,\ g_h \in [(J^{(i)})^{-1}]_{-1}$ by \autoref{lemma:tracefractional} and \autoref{cool}~(2).
  Thus we have
  $x f_h \in J^{(i+1)}$ and $y'g_h\in (J^{(i+1)})^{-1}$
  and so $xy'\in \tr_{R^{(i+1)}}(J^{(i+1)})$, as desired.
\end{itemize}
\item[\rm (2)]
Assume that \(\indeg_R(J)<0\).
We distinguish two cases according to the value of \(\rt_{\kk[R_1]}(J)\).

\smallskip

\noindent
\textbf{Case 1:} \(\rt_{\kk[R_1]}(J) \ge 0\).
In this case, for \(\alpha\in\{i,i+1\}\) we have
$J^{(\alpha)} = R^{(\alpha)}J_0 + R^{(\alpha)}J_\alpha$.

Indeed, fix \(n\ge 1\). (Note that \(\alpha>0\) since \(\alpha \ge K \ge |\indeg_R(J)|>0\).)
Then \((n-1)\alpha \ge 0\), and since \(\alpha \ge K \ge \rt_{\kk[R_1]}(J)\),
applying \autoref{rem:buruburu}~(2) yields
$J_{n\alpha}=R_{(n-1)\alpha}J_\alpha=[R^{(\alpha)}]_{n-1}J_\alpha$.
Hence \(\bigoplus_{n>0} J_{n\alpha}=R^{(\alpha)}J_\alpha\), and therefore
$J^{(\alpha)} = J_0 \oplus \bigoplus_{n>0} J_{n\alpha}
= R^{(\alpha)}J_0 + R^{(\alpha)}J_\alpha$.

\textbf{Case 2:} \(\rt_{\kk[R_1]}(J)<0\).
In this case, for \(\alpha\in\{i,i+1\}\) we have \(J^{(\alpha)}=R^{(\alpha)}J_0\).
Indeed, for any \(n\ge 0\) we have \(n\alpha \ge 0\), and since \(0 \ge \rt_{\kk[R_1]}(J)\),
applying \autoref{rem:buruburu}~(2) yields
$J_{n\alpha}=R_{n\alpha}J_0=[R^{(\alpha)}]_n J_0$.
Hence \(J^{(\alpha)}=R^{(\alpha)}J_0\).

In particular, note that in both Case~1 and Case~2 we have
\(\indeg_{R^{(i)}}(J^{(i)})=\indeg_{R^{(i+1)}}(J^{(i+1)})=0\).
Finally, assume that \(\rt_{\kk[R_1]}(J)\ge 0\) and consider the following two cases.
\begin{itemize}
  \item[(c)] If \( \tr_{R^{(i)}}(J^{(i)}) = \MM_{R^{(i)}} \), then we have \( y' \in \mm_{R^{(i)}} = J^{(i)} \cdot (J^{(i)})^{-1} \) and we can write
  $y' = \sum_{k} f_k g_k$ with $f_k \in [J^{(i)}]_0(=[J^{(i+1)}]_0=J_0)$ and $\ g_k \in [(J^{(i)})^{-1}]_1$ by \autoref{cool}~(2).
  Then
  we obtain \( (x g_k)a \in R^{(i+1)} \) for any \( a \in J_0 \bigoplus J_i \) because \( g_k a \in R \).
  Thus we have \( x g_k \in (J^{(i+1)})^{-1} \) for each \( k \) because $J^{(i+1)} = R^{(i+1)}J_0 + R^{(i+1)}J_{i+1}$
  and so
  $xy' = \sum_{k} f_k (xg_k) \in J^{(i+1)} \cdot (J^{(i+1)})^{-1}$.
  \item[(d)] If \( \tr_{R^{(i)}}(J^{(i)}) = {R^{(i)}} \), then
  we have
  $1 = \sum_{h} f_h g_h$ with $f_h \in [J^{(i)}]_0=[J^{(i+1)}]_0$ and $g_h \in [(J^{(i)})^{-1}]_0$ by \autoref{cool}~(2).
In this case, by definition we have \(x f_h \in J_1 = [J^{(i+1)}]_0 \subseteq J^{(i+1)}\).
Moreover, since \(g_h y' \in R_0\), we have \(g_h y' \in (J^{(i+1)})^{-1}\).
Therefore,
we have $xy'=\sum_h (x f_h)(g_h y') \in J^{(i+1)}\cdot (J^{(i+1)})^{-1}$,
as desired.
\end{itemize}
If \(\rt_{\kk[R_1]}(J)<0\), then the same argument as in \emph{(c)} and \emph{(d)} shows that
\(xy' \in J^{(i+1)}\cdot (J^{(i+1)})^{-1}\).
\end{itemize}
In all cases, we conclude that
$xy' \in J^{(i+1)}\cdot (J^{(i+1)})^{-1}$,
as desired.
\end{proof}

We will discuss the following setup:



\begin{setup}\label{atode}
Suppose that $R$ is a Cohen--Macaulay and generically Gorenstein semi-standard graded ring with $\dim(R)>0$.
Moreover, assume that $[\omega_R]_{-a_R}$ contains an $R$-regular element of $R$.
\end{setup}

This setup deals with cases that are more general than domains, as the following shows.

\begin{remark}\label{rem:PLEASECHECKBHCHAPTER4}
Suppose that $R$ is a Cohen--Macaulay semi-standard graded ring with $\dim(R)>0$.
If $R$ is a domain, or a generically Gorenstein and level ring, then there exists a homogeneous $R$-regular element in $[\omega_R]_{-a_R}$.
\begin{proof}
Note that every domain is generically Gorenstein, so the statement is straightforward when $R$ is a domain. Moreover, if $R$ is generically Gorenstein and level, then the statement follows because the proof for standard graded rings in \cite[Theorem 4.4.9]{bruns1998cohen} can be extended to semi-standard graded rings.
\end{proof}
\end{remark}

\begin{remark}
Assume \autoref{atode}.
Since \( \omega_R \) can be identified with a graded fractional ideal \( J_R \), we will regard it as such inside \( Q(R) \).
As \( R \) admits a homogeneous non-zero divisor of degree one, we may assume that \( \operatorname{indeg}(J_R) = \operatorname{indeg}(\omega_R) = -a_R \).
Furthermore, we may assume that \( [J_R]_{-a_R} \) contains an \( R \)-regular element.
\end{remark}

\begin{theorem}\label{prop:trace_stability}
Assume \autoref{atode}.
If \( R^{(K)} \) is nearly Gorenstein for some
$$K \ge \max\left\{ 1, \;|a_R|,\; \rt_{\kk[R_1]}(\omega_R),\; \rt_{\kk[R_1]}(R) \right\},$$ then so is \( R^{(i)} \) for all \( i \ge K \).
\end{theorem}
\begin{proof}
We may assume $\kk$ is infinite by \autoref{rem:PUNC}~(2).
Moreover, we may assume that $\omega_R$ is isomorphic to a fractional ideal $J$.
Set $A=R^{(i)}$.
Note that \(\omega_A \cong J^{(i)}\) by \cite[Corollary~(3.1.3)]{goto1978graded}, and that $|a_R|=|\indeg_R(\omega_R)|$ by definition.
The claim then follows from \autoref{cont}.
\end{proof}

\begin{corollary}\label{nicenice}
Assume \autoref{atode} and $R$ is standard graded.  
If \( R^{(K)} \) is nearly Gorenstein for some \( K \ge \rt_{R}(\omega_R) \), then so is \( R^{(i)} \) for all \( i \ge K \).
\end{corollary}

\begin{theorem}\label{thm:SUPERNICE}
Assume \autoref{atode}.
If $R$ satisfies condition~$(\natural)$, then there exists a positive integer $k_R$ such that $R^{(k)}$ is nearly Gorenstein for all $k \ge k_R$.
More specifically, set
$b = \rt_{\kk[R_1]}(R)$,
$J = [\tr_{R^{(b)}}(\omega_{R^{(b)}})]_1 R^{(b)}$.
Then $k_R = bj_R$ satisfies the desired condition,
where $j_R$ is the smallest integer $j$ satisfying
$$bj \ge \max \left\{ 1,\; |a_R|\;,\rt_{\kk[R_1]}(\omega_R),\; \left(1 + a_{R^{(b)} / J} \right)b \right\}.$$
\begin{proof}
Notice that $R^{(b)}$ satisfies $(\natural)$ by \autoref{GREAT}~(1).
Since $j_R > a_{R^{(b)}/J}$,
we have $R ^{(k_R)}=(R ^{(b)})^{(j_R)}$ is nearly Gorenstein by \autoref{GREAT}~(2).
Thus, since $k_R \ge \max \{
1,\; |a_R|,\;
\rt_{\kk[R_1]}{(\omega_R)},\; b\}$ and $R ^{(k_R)}$ is nearly Gorenstein,
we have $R ^{(k)}$ is nearly Gorenstein for any $k \ge k_R$ by \autoref{prop:trace_stability}.
\end{proof}
\end{theorem}

With the above preparations in place, we will now prove the final main result of this section.

\begin{corollary}
\label{cor:SUPERNICE}
Assume that 
$R$ is a Cohen--Macaulay semi-standard graded ring,
and $R$ is either a domain or a level ring.
If \( R \) satisfies $(\natural)$,
then there exists an integer \( k_R > 0 \) such that \( R^{(k)} \) is nearly Gorenstein for all \( k \ge k_R \).
\end{corollary}
\begin{proof}
We note that the zero ring is conventionally regarded as Gorenstein.
When $\dim(R)= 0$, we see that $R^{(k)} = 0$ for any $k > a_R$, so it suffices to prove the assertion for $\dim(R)> 0$.
Then $R$ is generically Gorenstein since $R$ satisfies $(\natural)$ with $\dim(R)>0$.
Therefore, since $R$ satisfies the condition of \autoref{atode} by \autoref{rem:PLEASECHECKBHCHAPTER4}, the statement follows from \autoref{thm:SUPERNICE}.
\end{proof}

\begin{example}\label{ex.thmC}
Let \( S = \langle (0,1), (3,1), (6,1), (9,1), (2,10) \rangle \subseteq \mathbb{N}^2 \), and consider the semigroup ring \( R = \mathbb{Q}[S] \). This defines a semi-standard graded affine semigroup ring with grading given by  
\[
\deg(t) = \deg(s^3t) = \deg(s^6t) = \deg(s^9t) = 1, \quad \deg(s^2t^{10}) = 10.
\]
Note that \( \mathrm{rt}_{\mathbb{Q}[R_1]}(R) = 10 \).
A computation using \texttt{Macaulay2}~\cite{M2} verifies that \( R \) is a level ring satisfying \((\natural)\). Moreover, we have
$\rt_{\QQ[R_1]}(\omega_R)=-9$
and $a_{A} = 2$,
where $A=R^{(10)} \big/ \mathrm{tr}_{R^{(10)}}(\omega_{R^{(10)}})$.
It then follows from \autoref{thm:SUPERNICE} that \( R^{(k)} \) is nearly Gorenstein for all \( k \geq  30\).
\end{example}

\section{Characterization of {$(\natural)$} in Affine Semigroup Rings and Ehrhart Rings}\label{sect_D}

In this section, we study how condition~$(\natural)$ manifests in affine semigroup rings and Ehrhart rings.
In particular, we prove that the property $(\natural)$ introduced in this paper agrees, for Ehrhart rings, with the property considered in~\cite{hall2023nearly}, and we give an alternative proof of \cite[Theorem~20]{hall2023nearly}.
First, we characterize condition $(\natural)$ for affine semigroup rings.

\begin{proposition}\label{prop:NICE}
Let $S \subseteq \ZZ^d$ be a pointed affine semigroup and let $R = \kk[S]$ be a Cohen--Macaulay semi-standard graded affine semigroup ring. Then $R$ satisfies $(\natural)$ if and only if $\tr_R(\omega_R) \supseteq (\bfx^{\bfe} : \bfe \in E_S)R$.
\begin{proof}
If $\tr_R(\omega_R) \supseteq (\bfx^{\bfe} : \bfe \in E_S)R$, then $R$ satisfies $(\natural)$ because $(\bfx^{\bfe} : \bfe \in E_S)R$ is an $\mm_R$-primary ideal generated by degree-one elements by \cite[Proposition 3.4]{miyashita2023comparing}.

Conversely, suppose $R$ satisfies $(\natural)$. First, consider the case where $\kk$ is infinite. Then there exists a linear system of parameters $I = (\theta_1, \ldots, \theta_{\dim R})$ with $\theta_i \in R_1$ such that $\tr_R(\omega_R) \supseteq I$ by \autoref{rem:PUNC}~(3). Since $\omega_R$ is $\ZZ^d$-graded, it follows from \autoref{Mult} that $\tr_R(\omega_R)$ is also $\ZZ^d$-graded. Hence, every $\ZZ^d$-homogeneous element of $I$ lies in $\tr_R(\omega_R)$, and by \autoref{para}, we obtain $\tr_R(\omega_R) \supseteq (\bfx^{\bfe} : \bfe \in E_S)R$.

Now suppose $\kk$ is arbitrary. Let $\KK' = \kk(x)$ be an infinite field extension and set $A = R \otimes_\kk \KK'$. Then $A$ satisfies $(\natural)$ by \autoref{rem:PUNC}~(4), so $\tr_{A}(\omega_{A}) \supseteq (\bfx^{\bfe} : \bfe \in E_S)A$. By \autoref{rem:PUNC}~(1), $\tr_{A}(\omega_{A}) = \tr_R(\omega_R)A$. Since both ideals are $\ZZ^d$-graded by \autoref{Mult}, it follows that $\tr_R(\omega_R) \supseteq (\bfx^{\bfe} : \bfe \in E_S)R$.
\end{proof}
\end{proposition}

Next, after recalling the basic definitions of Ehrhart rings, we characterize condition $(\natural)$ in terms of the property considered in~\cite{hall2023nearly}.
We recall some definitions about Ehrhart rings arising from lattice polytope.
As in the previous subsection,
we fix an integer $d > 0$ and a field $\kk$.
We denote the natural pairing between an element $n \in (\RR^d)^*$ and an element $x \in \RR^d$ by $n(x)$.
Throughout this subsection, let $P \subseteq \RR^d$ be a lattice polytope, $\facets(P)$ the set of facets of $P$, and $\verto(P)$ the set of vertices of $P$.
Moreover, recall that we always assume $P$ is full-dimensional and has the facet presentation
\[
P = \setcond{x \in \RR^d}{n_F(x) \ge -h_F \text{ for all $F \in \facets(P)$}},
\]
where each height $h_F$ is an integer and each inner normal vector $n_F \in (\ZZ^d)^*$ is a \emph{primitive} lattice point, i.e., a lattice point such that the greatest common divisor of its coordinates is $1$.

Let $C_P$ be the \textit{cone over $P$}, that is,
\[
C_P = \RR_{\ge 0}(P \times \set{1}) = \setcond{(x,k) \in \RR^{d+1}}{n_F(x) \ge -k h_F \text{ for all $F \in \facets(P)$}}.
\]

\begin{definition}\label{def:Ehrhartring}
The \textit{Ehrhart ring} of $P$ is defined as
$A(P)
= \kk[\tb^x s^k : k \in \NN \text{ and } x \in kP \cap \ZZ^d]$,
where $\tb^x = t_1^{x_1} \cdots t_d^{x_d}$ and $x = (x_1, \ldots, x_d) \in kP \cap \ZZ^d$.
\end{definition}

Notice that the Ehrhart ring of $P$ is a normal affine semigroup ring, and hence it is Cohen--Macaulay.
Moreover, we can regard $A(P)$ as an positively graded ring by setting $\deg(\tb^x s^k) = k$ for each $x \in kP \cap \ZZ^d$.
It is known that $A(P)$ is semi-standard graded.

We also define another affine semigroup ring, the \textit{toric ring} of $P$, as
$\kk[P] = \kk[\tb^x s : x \in P \cap \ZZ^d]$.
The toric ring of $P$ is a standard ring by setting
$\deg(\tb^x s) = 1$ for each $x \in P \cap \ZZ^d$.
It is known that $\kk[P] = A(P)$ if and only if $P$ has the integer decomposition property.
Here, we say that $P$ has the \emph{integer decomposition property} (i.e., $P$ is \emph{IDP}) if for all positive integers $k$ and all $x \in kP \cap \ZZ^d$, there exist $y_1, \ldots, y_k \in P \cap \ZZ^d$ such that $x = y_1 + \cdots + y_k$.

In order to describe the canonical module and the anti-canonical module of $A(P)$ in terms of $P$, we prepare some notation.
For a polytope or cone $\sigma$, we denote the strict interior of $\sigma$ by $\into(\sigma)$.
Notice that
\[
    \into(C_P) = \setcond{(x,k) \in \RR^{d+1}}{n_F(x) > -k h_F \text{ for all $F \in \facets(P)$}}.
\]
Moreover, we define
\[
\ant(C_P) := \setcond{(x,k) \in \RR^{d+1}}{ n_F(x) \ge -k h_F - 1 \text{ for all $F \in \facets(P)$}}.
\]

\begin{proposition}[see {\cite[Proposition 4.1 and Corollary 4.2]{herzog2019measuring}}]\label{prop:can_antican}
The canonical module of $A(P)$ and the anti-canonical module of $A(P)$ are given respectively by
\[
\omega_{A(P)} = \left( \tb^x s^k : (x,k) \in \into(C_P) \cap \ZZ^{d+1} \right) A(P),\quad
\omega_{A(P)}^{-1} = \left( \tb^x s^k : (x,k) \in \ant(C_P) \cap \ZZ^{d+1} \right) A(P).
\]
Furthermore, the negated $a$-invariant of $A(P)$ coincides with the codegree of $P$, i.e.,
\[
a_{A(P)} = -\min\setcond{k \in \ZZ_{\ge 1}}{\strint(kP) \cap \ZZ^d \neq \varnothing}.
\]
\end{proposition}

Let $A$ and $B$ be subsets of $\RR^d$.
Their \emph{Minkowski sum} is defined as
$A + B := \setcond{x + y}{x \in A,\, y \in B}$.
We recall that the \emph{(direct) product} of two polytopes $P \subseteq \RR^d$ and $Q \subseteq \RR^e$ is denoted by $P \times Q \subseteq \RR^{d+e}$.
Let $P$ and $Q$ be two lattice polytopes.
It is known that $\kk[P \times Q]$ is isomorphic to the Segre product $\kk[P] \natural \kk[Q]$.

For a subset $X$ of $\RR^{d+1}$ and $k \in \ZZ$, we define the $k$-th \emph{piece} of $X$ by
\[
X_k := \{x \in \RR^d : (x,k) \in X\}.
\]
Moreover, for a lattice polytope $P$, we denote its \emph{codegree} by $a_P$ (see \autoref{prop:can_antican}).

\begin{definition}[{see \cite[Definition 15]{hall2023nearly}}]\label{def:floor_rem}
Let $P \subseteq \RR^d$ be a lattice polytope with codegree $a_P$.
We set
\[
\lfloor P\rfloor := \conv(\strint(C_P)_{a_P} \cap \ZZ^d), \qquad
\set{P} := \conv(\ant(C_P)_{1 - a_P} \cap \ZZ^d).
\]
\end{definition}

\begin{remark}
In \cite[Definition 15]{hall2023nearly}, for a lattice polytope \( P \subseteq \mathbb{R}^d \), the floor polytope \( \lfloor P \rfloor \) is defined as
\[
\lfloor P \rfloor := \operatorname{conv}(\ant(C_P)_1 \cap \mathbb{Z}^d),
\]
and is referred to as the \emph{floor polytope}.
This definition differs from the one given in \autoref{def:floor_rem}. For simplicity of notation, we adopt the definition in \autoref{def:floor_rem} throughout this paper.
\end{remark}

The following is a restatement of results from \cite{hall2023nearly} using our notation $\lfloor P\rfloor$.

\begin{remark}[{\cite[Lemma 16]{hall2023nearly}}]\label{lem:flo_rem}
Let $P \subseteq \RR^d$ be a lattice polytope with codegree $a_P$.
Then:
\begin{enumerate}
    \item $\lfloor P\rfloor \subseteq \setcond{x \in \RR^d}{n_F(x) \ge 1 - a_P h_F \text{ for all } F \in \facets(P)}$;
    \item $\set{P} \subseteq \setcond{x \in \RR^d}{n_F(x) \ge (a_P - 1) h_F - 1 \text{ for all } F \in \facets(P)}$;
\end{enumerate}
\end{remark}

Let \( P \subseteq \mathbb{R}^d \) be a lattice polytope, and let \( C_P \) be the cone over \( P \).  
Let \( \verto(P) \) denote the set of vertices of \( P \).  
Moreover, let \( V_P \subseteq C_P \) denote the union of all extremal rays of \( C_P \), that is,
\[
V_P = \bigcup_{v \in \verto(P)} \mathbb{R}_{\ge 0}(v,1) \subseteq \mathbb{R}^{d+1}.
\]
Note that \( (V_P)_1 = \verto(P) \).

\begin{lemma}\label{prop:nearly}
Let \( P \subseteq \RR^d \) be a lattice polytope with codegree \( a_P \), and set \( S = C_P \cap \ZZ^{d+1} \). Then, the inclusion
$\tr_{A(P)}(\omega_{A(P)}) \supseteq (\xb^{\eb} : \eb \in E_S)A(P)$
holds if and only if the following inclusion holds:
\begin{equation}\label{eq:nG_cones}
(V_P \cap \ZZ^{d+1}) \setminus \{0\} \subseteq \strint(C_P) \cap \ZZ^{d+1} + \ant(C_P) \cap \ZZ^{d+1}.
\end{equation}
In particular, if the above inclusion holds, then we have
\[
\verto(P) \subseteq \strint(C_P)_{a_P} \cap \ZZ^d + \ant(C_P)_{1 - a_P} \cap \ZZ^d.
\]
\end{lemma}
\begin{proof}
This can be proved in essentially the same way as \cite[Proposition~14]{hall2023nearly}.
\end{proof}

\begin{proposition}\label{prop:dec}
Let $P \subseteq \RR^d$ be a lattice polytope and set $S = C_P \cap \ZZ^{d+1}$. Then $\tr_{A(P)}(\omega_{A(P)}) \supseteq (\xb^{\eb} : \eb \in E_S)A(P)$ if and only if $P = \lfloor P \rfloor + \{P\}$.
\end{proposition}
\begin{proof}
Notice that the inclusion $P \supseteq \lfloor P \rfloor + \{P\}$ always holds. Indeed, for $x \in \lfloor P \rfloor$ and $y \in \{P\}$, \autoref{lem:flo_rem} implies $n_F(x + y) \ge -h_F$ for all $F \in \facets(P)$, hence $x + y \in P$.

Now suppose $\tr_{A(P)}(\omega_{A(P)}) \supseteq (\xb^{\eb} : \eb \in E_S)A(P)$. Since $\lfloor P \rfloor + \{P\}$ is a convex polytope, it suffices to show that $\verto(P) \subseteq \lfloor P \rfloor + \{P\}$. Let $v \in \verto(P)$. Then we can write $v = x + y$ with $x \in \lfloor P \rfloor \cap \ZZ^d$ and $y \in \{P\} \cap \ZZ^d$ by \autoref{prop:nearly}. Hence $P \subseteq \lfloor P \rfloor + \{P\}$, and equality follows.

Conversely, assume $P = \lfloor P \rfloor + \{P\}$. Note that $(\xb^{\eb} : \eb \in E_S)A(P)$ is generated by monomials $\tb^x s$ with $x \in \verto(P)$. Let $x \in \verto(P)$, then $x = x_1 + x_2$ with $x_1 \in \lfloor P \rfloor \cap \ZZ^d$ and $x_2 \in \{P\} \cap \ZZ^d$. Then $\tb^{x_1}s^{a_P} \in \omega_{A(P)}$ and $\tb^{x_2}s^{1 - a_P} \in \omega_{A(P)}^{-1}$, so
$\tb^x s = (\tb^{x_1}s^{a_P})(\tb^{x_2}s^{1 - a_P}) \in \omega_{A(P)} \cdot \omega_{A(P)}^{-1}$.
\end{proof}

\begin{corollary}\label{thm:goodnews}
Let $P \subseteq \RR^d$ be a lattice polytope. Then $A(P)$ satisfies $(\natural)$ if and only if $P = \lfloor P \rfloor + \{P\}$.
\begin{proof}
It follows from \autoref{prop:NICE} and \autoref{prop:dec}.
\end{proof}
\end{corollary}

\begin{corollary}[{\cite[Theorem 20]{hall2023nearly}}]
Let $P \subseteq \RR^d$ be a lattice polytope satisfying $P = \lfloor P\rfloor + \{P\}$. Then there exists some integer $K > 0$ such that $A(kP)$ is nearly Gorenstein for any $k \ge K$.
\begin{proof}
Since $A(kP) \cong A(P) ^{(k)}$ for any $k \in \ZZ_{>0}$, it follows from \autoref{cor:SUPERNICE} and \autoref{thm:goodnews}.
\end{proof}
\end{corollary}

\section*{Acknowledgments}
I would like to thank my supervisor, Akihiro Higashitani, for helpful discussions.
I also thank Shinya Kumashiro, Koji Matsushita, and Kohji Yanagawa for their valuable comments.
In particular, I am indebted to Shinya Kumashiro for explaining to me the essential part of the proof of \autoref{NZDLOVE}.
Finally, I would like to thank the anonymous referees for their helpful suggestions that improved this paper.


\begin{thebibliography}{100}
\bibitem{bourbaki1998commutative}
Nicolas Bourbaki.
\newblock {\em Commutative Algebra: Chapters~1--7}, volume~1.
\newblock Springer, 1998.

\bibitem{bruns1998cohen}
Winfried Bruns and J{\"u}rgen Herzog.
\newblock {\em Cohen--Macaulay Rings}.
\newblock Cambridge Studies in Advanced Mathematics, No.~39.
\newblock Cambridge Univ. Press, 1998.

\bibitem{ficarra2024canonical!}
Antonino Ficarra, J{\"u}rgen Herzog, Dumitru~I. Stamate, and Vijaylaxmi Trivedi.
\newblock The canonical trace of determinantal rings.
\newblock {\em Arch. Math. (Basel)}, 123(5):487\ndash 497, 2024.

\bibitem{goto1978graded}
Shiro Goto and Keiichi Watanabe.
\newblock On graded rings, I.
\newblock {\em J. Math. Soc. Japan}, 30(2):179--213, 1978.

\bibitem{M2}
Daniel~R. Grayson and Michael~E. Stillman.
\newblock Macaulay2, a software system for research in algebraic geometry.
\newblock Available at \url{http://www2.macaulay2.com}.

\bibitem{hall2023nearly}
Thomas Hall, Max K{\"o}lbl, Koji Matsushita, and Sora Miyashita.
\newblock Nearly Gorenstein polytopes.
\newblock {\em Electron. J. Combin.}, 30(4), 2023.

\bibitem{herzog2019trace}
J{\"u}rgen Herzog, Takayuki~Hibi, and Dumitru~I. Stamate.
\newblock The trace of the canonical module.
\newblock {\em Israel J. Math.}, 233:133--165, 2019.

\bibitem{herzog2019measuring}
J{\"u}rgen Herzog, Fatemeh Mohammadi, and Janet Page.
\newblock Measuring the non-Gorenstein locus of Hibi rings and normal affine semigroup rings.
\newblock {\em J. Algebra}, 540:78--99, 2019.

\bibitem{matsumura1989commutative}
Hideyuki Matsumura.
\newblock {\em Commutative Ring Theory}.
\newblock Cambridge Studies in Advanced Mathematics, No.~8.
\newblock Cambridge Univ. Press, 1989.

\bibitem{miller2005combinatorial}
Ezra Miller and Bernd Sturmfels.
\newblock {\em Combinatorial Commutative Algebra}.
\newblock Graduate Texts in Mathematics, Vol.~227.
\newblock Springer, 2005.

\bibitem{miyashita2023comparing}
Sora Miyashita.
\newblock Comparing generalized Gorenstein properties in semi-standard graded rings.
\newblock {\em J. Algebra}, 647 (2024):823--843.

\bibitem{miyashita2024canonical}
Sora Miyashita and Matteo Varbaro.
\newblock The canonical trace of Stanley--Reisner rings that are {G}orenstein on the punctured spectrum.
\newblock {\em Int. Math. Res. Not. IMRN} 2025 (2025), no.~12, rnaf176.
\newblock DOI: 10.1093/imrn/rnaf176.

\bibitem{miyashita2025pseudo}
Sora Miyashita.
\newblock When do pseudo-{G}orenstein rings become {G}orenstein?
\newblock {\em Bull. Lond. Math. Soc.} (2025), Article~blms.70186.
\newblock DOI: 10.1112/blms.70186.



\bibitem{stanleyhilbert}
Richard~P. Stanley.
\newblock Hilbert functions of graded algebras*.
\newblock {\em Adv. Math.}, 28:57--83, 1978.

\bibitem{stanley1991hilbert}
Richard~P. Stanley.
\newblock On the Hilbert function of a graded Cohen--Macaulay domain.
\newblock {\em J. Pure Appl. Algebra}, 73(3):307--314, 1991.

\bib{stanley1991f}{article}{
  author  = {Stanley, Richard~P.},
  title   = {$f$-vectors and $h$-vectors of simplicial posets},
  journal = {J. Pure Appl. Algebra},
  volume  = {71},
  number  = {2--3},
  pages   = {319\ndash 331},
  date    = {1991},
}



\end{thebibliography}
\end{document}